\pdfoutput=1
\documentclass[11pt,leqno]{article}

\usepackage{mathptmx}
\usepackage{amsmath,amsthm,amssymb}
\usepackage{graphicx,subfigure,booktabs,enumerate} 

\usepackage{color}

\theoremstyle{plain}
\newtheorem{thm}{Theorem}[section]

\newtheorem{lem}[thm]{Lemma}

\theoremstyle{remark}
\newtheorem{rem}[thm]{Remark}

\theoremstyle{definition}
\newtheorem{ex}[thm]{Example}

\newtheorem{assum}[thm]{Assumption}

\makeatletter

\numberwithin{equation}{section}
\numberwithin{figure}{section}
\numberwithin{table}{section}

\title{\Large\bf Convergence of meshfree collocation \\
methods for fully nonlinear parabolic equations}
\author{Yumiharu Nakano\\[1em]
        \small{Department of Mathematical and Computing Science, School of Computing} \\
        \small{Institute of Science Tokyo} \\
        \small{2-12-1 W8-28 Ookayama 152-8550, Tokyo, Japan} \\
		\small{e-mail: nakano@comp.isct.ac.jp}
}

\date{\today}

\begin{document}

\maketitle

\begin{abstract}
We prove the convergence of meshfree collocation methods for 
the terminal value problems of fully 
nonlinear parabolic partial differential equations in the framework of 
viscosity solutions, provided that the basis function 
approximations of the terminal condition and the nonlinearities 
are successful at each time step. 
A numerical experiment with a radial basis function  
demonstrates the convergence property. 

\begin{flushleft}
\end{flushleft}
\begin{flushleft}
\end{flushleft}
\end{abstract}





\section{Introduction}\label{sec:1}

In this paper, we are concerned with the numerical methods 
for the terminal value problems of the parabolic partial differential equations: 
\begin{equation}
\label{eq:1.1}
\left\{
\begin{split}
 &-\partial_t v + F(t,x,v(t,x),Dv(t,x),D^2v(t,x))=0, \quad (t,x)\in 
 [0,T)\times \mathbb{R}^d, \\
 &v(T,x)=f(x), \quad x\in\mathbb{R}^d, 
\end{split}
\right. 
\end{equation}
where $F:[0,T]\times\mathbb{R}^d\times\mathbb{R}\times\mathbb{R}^d\times
\mathbb{S}^d\to\mathbb{R}$, and 
$\mathbb{S}^d$ stands for the totality of symmetric $d\times d$ real matrices. 
Here we have denoted  
by $\partial_t$ the partial differential 
operator with respect to the time variable $t$,  
by $D^j\equiv D^j_x$ the $j$-th order partial differential 
operator with respect to the space variable $x$. 
The conditions imposed on the function $F$ are described in 
Section \ref{sec:2} below. 
The terminal value problem (\ref{eq:1.1}) mainly appears from 
probabilistic problems. In linear cases the solution to (\ref{eq:1.1}) is given by 
the expectation of a diffusion process, 
whereas in nonlinear cases of Hamilton-Jacobi-Bellman type, the 
solution is given by the value function of a stochastic control problem. 

Existing numerical methods applicable to (\ref{eq:1.1}) are the finite difference methods 
(see, e.g., Kushner and Dupuis \cite{kus-dup:2001} and 
Bonnans and Zidani \cite{bon-zid:2003}), the finite-element like methods 
(see, e.g., Camilli and Falcone \cite{cam-fal:1995} and 
Debrabant and Jakobsen \cite{deb-jak:2012}), 
and the probabilistic methods (see, e.g., Pag{\`e}s et al.~\cite{pag-pha-pri:2004}, 
Fahim et al.~\cite{fah-tou-war:2011} and Nakano \cite{nak:2014b}). 
It should be mentioned that 
these methods have difficulties in applying to the problems with high-dimensional 
state space, which appear as an application of (\ref{eq:1.1}). 
For examples, in the finite difference methods, 
the diffusion matrix in the Hamiltonian should basically be diagonally dominant
for ensuring its convergence 
(see, e.g., \cite{kus-dup:2001}). 
Also, the finite-element like methods require the interpolation of the solutions 
in the state space that preserve a monotonicity condition, 
and need involved computational procedures 
for the implementation in high-dimensional problems 
(see Carlini et al.~\cite{car-fal-fer:2004}). 

An another possible approach to (\ref{eq:1.1}) is to use 
the meshfree collocation method proposed by Kansa \cite{kan:1990b}. 
In this method, we seek an approximate solution 
of the form of a linear combination of a radial basis function 
(e.g., multiquadrics in the Kansa's original work). Substituting this form into 
a partial differential equation leads to an equation for the collocation points. 
Then the approximate solution is constructed by the meshfree interpolation 
of these collocation points. In general,  
this procedure allows for a simpler numerical implementation compared to the 
finite-element like methods, and it needs 
less computational time compared to the probabilistic methods. 
As for the convergence, 
rigorous analyses have been done for linear equations. 
See Chapter 15 in Wendland \cite{wen:2010}, 
Schaback \cite{sch:2010}, Lee et al.~\cite{lee-etal:2009}, 
Ling and Schaback \cite{lin-sch:2008}, 
and the references therein. In nonlinear cases,  
Huang et al.~\cite{hua-etal:2006} numerically shows the convergence in the case of  
a Hamilton-Jacobi-Bellman equation of the first order, a special case of (\ref{eq:1.1}). 
However, to the best of our knowledge, the rigorous convergence issue 
for the nonlinear parabolic equations (\ref{eq:1.1}) has not been addressed in the literature.

In this paper, we present a generalization of Kansa's collocation method 
and prove its rigorous convergence for 
the nonlinear parabolic equations (\ref{eq:1.1}). 
In doing so, we consider solutions of (\ref{eq:1.1}) in the viscosity sense 
since the smoothness of solutions cannot be expected in our nonlinear cases. 
In this framework, it is known that the abstract method proposed by 
Barles and Souganidis \cite{bar-sou:1991} is a powerful tool for checking 
the convergence of a given family of functions to a unique viscosity solution. 
Roughly speaking, if an operator that constructs the possible approximate solution 
 has monotonicity, stability, and consistency properties, then by the arguments in 
\cite{bar-sou:1991} we can basically prove its convergence.  
In our case, however, this technique cannot be applied in a trivial way 
since the collocation method includes the derivative terms and thus 
violates the monotonicity condition.  
We find that a key to overcoming this difficulty is 
Lemma 4.1 in Kohn and Serfaty \cite{koh-ser:2010}. 
Using this lemma, they show that an approximation scheme with a max-min representation 
has the consistency property. The statement of this lemma, however, suggests that 
its converse is also true, i.e.,  
every smooth consistent method has the max-min representation with a negligible term 
and so has the monotonicity in an approximation sense,  
since their max-min representation is approximately monotone.  
Therefore our task is to justify this observation in our situation. 

The present paper is organized as follows. 
In Section \ref{sec:2}, we briefly review the meshfree interpolation theory and 
derive a general collocation method for (\ref{eq:1.1}). 
We rigorously state our assumptions and prove the convergence property 
in Section \ref{sec:3}. 
Section \ref{sec:4} exhibits a numerical example.

\section{Generalization of Kansa's method}\label{sec:2}

Throughout this paper, for $a=(a_i)\in\mathbb{R}^{\ell}$ and 
$\tilde{a}\in\mathbb{R}^{\ell_1\times\ell_2}$, 
we write $|a|=(\sum_{i=1}^{\ell}a_{i}^2)^{1/2}$ and 
$|\tilde{a}|=\sup_{y\in\mathbb{R}^{\ell_2}\setminus\{0\}}|\tilde{a}y|/|y|$, respectively. 
We denote by $a^{\mathsf{T}}$ the transpose of a vector or matrix $a$. 
By $C$ we denote positive constants that may not be necessarily 
 equal with each other. 
We also write $C_{\kappa_1,\ldots,\kappa_{\ell}}$ for 
a positive constant $C$ depending only on parameters
$\kappa_1,\ldots,\kappa_{\ell}$. 
For a multiindex $\alpha=(\alpha_1,\ldots,\alpha_d)$ of nonnegative integers and 
a function $u$, we define $D^{\alpha}u(x)$ by the usual manner, i.e.,  
\begin{equation*}
 D^{\alpha}u(x)=\frac{\partial^{|\alpha|_1}u(x)}
  {\partial x_1^{\alpha_1}\cdots\partial x_d^{\alpha_d}}
\end{equation*}
with $|\alpha|_1=\alpha_1+\cdots +\alpha_d$. 
For $m\in\mathbb{N}\cup\{0\}$ we denote by $\Pi_m(\mathbb{R}^{\ell})$ 
the set of all $\mathbb{R}^{\ell}$-valued polynomial of degree at most $m$.

In this section, we describe a meshfree collocation method for 
(\ref{eq:1.1}), which is a generalization of Kansa's method in the 
parabolic cases. 
First, we briefly review the basis of the interpolation theory with 
conditionally positive definite kernels. 
We refer to Wendland \cite{wen:2010} for a complete account. 
In general, a meshfree method seeks an approximate function 
in the space spaned by a prespecified kernel. 
As the kernel we consider a smooth, symmetric conditionally positive definite kernel   
$\Phi: \mathbb{R}^d\times\mathbb{R}^d\to \mathbb{R}$ of order $m$. 
More precisely, $\Phi$ is assumed to satisfy the following: 
\begin{enumerate}[\rm (i)]
\item $\Phi\in C^{2\nu}(\mathbb{R}^d\times\mathbb{R}^d)$ for some 
 $\nu\ge 2$;    
\item $\Phi(x,y)=\Phi(y,x)$ for $x,y\in\mathbb{R}^d$;  
\item for every $\ell\in\mathbb{N}$, for all pairwise distinct 
$y_1,\ldots, y_{\ell}\in\mathbb{R}^{d}$ and for all 
$\alpha\in\mathbb{R}^{\ell}\setminus\{0\}$ satisfying  
\begin{equation}
\label{eq:2.1}
  \sum_{j=1}^{\ell}\alpha_j\pi(y_j)=0, \quad \pi\in\Pi_{m-1}(\mathbb{R}^d), 
\end{equation}
we have 
\begin{equation}
\label{eq:2.2}
  \sum_{i,j=1}^{\ell}\alpha_i\alpha_j\Phi(y_i,y_j)>0. 
\end{equation}
\end{enumerate}
If (\ref{eq:2.2}) holds without (\ref{eq:2.1}), then 
$\Phi$ is called a positive definite kernel. 

\begin{ex}
Here are some examples of the conditionally positive definite kernels. 
In each case, $\Phi$ is given by $\Phi(x,y)=\phi(|x-y|)$, where 
$\phi: [0,\infty)\to\mathbb{R}$, called a radial basis function (RBF). 
\begin{enumerate}[\rm (i)]
 \item Gaussian RBF: $\phi(r)=e^{-\alpha r^2}$, $r\ge 0$, 
  with $\alpha>0$. In this case, $\Phi$ is positive definite. 
 \item multiquadric RBF: $\phi(r)=(\alpha^2+r^2)^{\beta}$, $r\ge 0$, 
  with $\alpha\in\mathbb{R}$, 
  $\beta\in\mathbb{R}\setminus(\mathbb{N}\cup\{0\})$.  
  In this case, $\Phi$ is positive definite for $\beta<0$. 
\end{enumerate}
\end{ex}

Let $\Omega$ be a bounded open subset of $\mathbb{R}^d$. 
Suppose that we are in a position to compute a numerical solution of (\ref{eq:1.1}) on 
$\Omega$. Then assume that $\Omega$ satisfies an  
interior cone condition, i.e., 
there exists $\theta\in (0,\pi/2)$ and $r>0$ such that 
for any $x\in\Omega$, 
$$
 C(x,\zeta(x),\theta,r):=\left\{x+\lambda y: y\in\mathbb{R}^d, \; |y|=1, \; 
 y^{\mathsf{T}}\zeta(x)\ge \cos\theta, \; \lambda\in [0,r]\right\} 
 \subset\Omega
$$
holds for some $\zeta(x)\in\mathbb{R}^d$ with $|\zeta(x)|=1$.   

Let $X=\{x^{(1)},\cdots,x^{(N)}\}$ be a set of pairwise distinct points in 
$\Omega$.    
Let $\pi_1,\ldots,\pi_Q$ be a basis of $\Pi_{m-1}(\mathbb{R}^d)$, 
where $Q=\mathrm{dim}(\Pi_{m-1}(\mathbb{R}^d))=(m+d)!/(m!d!)$. 
Denote $P=(\pi_k(x^{(j)}))\in\mathbb{R}^{N\times Q}$ 
and $A_{\Phi,X}=\{\Phi(x^{(i)},x^{(j)})\}_{1\le i,j\le N}$. 
We assume that $X$ is a $\Pi_{m-1}(\mathbb{R}^d)$-unisolvent 
set, i.e., $\pi\in\Pi_{m-1}(\mathbb{R}^d)$ with $\pi(x)=0$ on $X$ must be zero 
polynomial. 
Then, it follows from \cite[Theorem 8.21]{wen:2010} that the system 
\begin{equation}
\label{eq:2.3}
\left(
\begin{matrix}
A_{\Phi,X} & P \\
P^{\mathsf{T}} & 0
\end{matrix}
\right)
\left(
\begin{matrix}
\xi \\ \eta
\end{matrix}\right)
= 
\left(
\begin{matrix}
b \\ 0
\end{matrix}
\right)
\end{equation}
has a unique solution $(\xi(b),\eta(b))\in\mathbb{R}^N\times\mathbb{R}^Q$ 
for any $b\in\mathbb{R}^N$. 
Thus, the function 
\begin{equation*}
I_{g,X}(x) = \sum_{j=1}^N\xi_j(g|_X)\Phi(x,x^{(j)}) 
 + \sum_{i=1}^Q\eta_i(g|_X)\pi_i(x), \quad x\in\Omega, 
\end{equation*}
that interpolates $g$ on $X$ becomes an approximation of $g$. 
Here, $\xi(b)=(\xi_1(b),\ldots,\xi_N(b))^{\mathsf{T}}$, 
$\eta(b)=(\eta_1(b),\ldots,\eta_Q(b))^{\mathsf{T}}$ for $b\in\mathbb{R}^N$, 
and we have set $g|_X=(g(x^{(1)}),\ldots,g(x^{(N)}))^{\mathsf{T}}$. 

\begin{rem}
\label{rem:2.3}
If $\Phi$ is positive definite, then 
the matrix $A_{\Phi,X}$ is invertible and for $b\in\mathbb{R}^N$ 
the solution of (\ref{eq:2.3}) is given by 
\begin{equation*}
 \xi(b) = A_{\Phi,X}^{-1}b, \quad \eta(b)=0. 
\end{equation*}
In particular, we can drop the polynomial term in the interpolation. 
\end{rem}
Next we recall the error estimation results for interpolation by 
conditionally positive definite kernels. 
Let $\mathcal{N}_{\Phi}(\Omega)$ be the native space corresponding to 
$\Phi$. See \cite{wen:2010} for a precise definition. 
Here, we remark that $\mathcal{N}_{\Phi}(\Omega)$ is a linear subspace of 
$C(\Omega)$ equipped with 
a semi-inner product $(\cdot,\cdot)_{\mathcal{N}_{\Phi}(\Omega)}$. 
If $g,g^{\prime}\in C(\Omega)$ are of the form 
\begin{equation*}
 g(x)=\sum_{j=1}^M\alpha_j\Phi(x,y_j), \quad 
 g^{\prime}(x)= \sum_{j=1}^M\alpha_j^{\prime}\Phi(x,y_j^{\prime}), 
 \quad x\in\Omega, 
\end{equation*}
where 
$M,M^{\prime}\in\mathbb{N}$, 
$\alpha,\alpha^{\prime}\in\mathbb{R}^N$, 
$y_1,\ldots,y_M, y_1^{\prime},\ldots,y^{\prime}_{M^{\prime}}\in\Omega$, 
with $\sum_{j=1}^M\alpha_j\pi(y_j)=
\sum_{j=1}^{M^{\prime}}\alpha_j^{\prime}\pi(y_j^{\prime})=0$  
for all $\pi\in\Pi_{m-1}(\mathbb{R}^d)$, then  
$$
 (g,g^{\prime})_{\mathcal{N}_{\Phi}(\Omega)}= 
 \sum_{j=1}^M\sum_{\ell=1}^{M^{\prime}}\alpha_j\alpha_{\ell}^{\prime}
 \Phi(y_j,y_{\ell}^{\prime}). 
$$

\begin{ex}
Suppose that $\Phi$ is given by $\Phi(x,y)=\phi(|x-y|)$ where 
$\phi$ is some function on $[0,\infty)$ such that   
$x\mapsto \phi(|x|)$ is integrable and has a Fourier transform that 
decays as $(1+|\cdot|^2)^{-k}$, $k\in\mathbb{N}$, $k>d/2$. 
Suppose moreover that 
$\Omega$ has a Lipschitz boundary.  
Then $\mathcal{N}_{\Phi}(\Omega)$ coincides with the $L^2$-Sobolev 
space on $\Omega$ of order $k$ with equivalent norms. 
\end{ex}

The error of the interpolation is estimated as follows: 
for every $g\in\mathcal{N}_{\Phi}(\Omega)$ and every multiindex 
$\alpha$ with $|\alpha|\le \nu$, 
\begin{equation}
\label{eq:2.4}
 |D^{\alpha}g(x)-D^{\alpha}I_{g,X}(x)|\le 
 C_{\nu,\Phi}\Delta^{\nu-|\alpha|}_{X,\Omega}|g|_{\mathcal{N}_{\Phi}(\Omega)}, 
 \quad x\in\Omega,   
\end{equation}
where $|\cdot|_{\mathcal{N}_{\Phi}(\Omega)}
=(\cdot,\cdot)_{\mathcal{N}_{\Phi}(\Omega)}^{1/2}$ and 
$\Delta_{\Omega,X}$ is the fill distance defined by 
$$
 \Delta_{\Omega,X} = \sup_{x\in\Omega}\min_{j=1,\ldots,N}
  |x-x^{(j)}|. 
$$

\begin{rem}
In the above, we have assumed that $\Omega$ satisfies an interior cone condition 
and $X$ is $\Pi_{m-1}(\mathbb{R}^d)$-unisolvent. 
Typical examples are the cases that $\Omega$ is star-shaped 
(see Proposition 11.26 in \cite{wen:2010}) and  
$X$ is a set of uniformly spaced grid points in $\Omega$ with $N\ge m$.  
\end{rem}

Now, let us describe the meshfree collocation methods for our parabolic equations. 
We start with the formal time discretization of (\ref{eq:1.1}) to get  
\begin{equation}
\label{eq:2.5}
 \frac{v(t_{k+1},x)-v(t_k,x)}{h} \simeq\theta F(t_{k+1},x; v(t_{k+1},\cdot)) 
  + (1-\theta)F(t_{k},x;v(t_k,\cdot)) 
\end{equation}
where $t_k=kh$, $k=0,\ldots,n$ and $h=T/n$, $\theta\in [0,1]$, and for 
any $\varphi\in C^2(\mathbb{R}^d)$ 
$$
 F(t,x;\varphi) = F(t,x,\varphi(x), D\varphi(x),D^2\varphi(x)), 
  \quad x\in\mathbb{R}^d. 
$$
Let us denote by $v_{k,j}$, $k=0,\ldots,n$, $j=1,\ldots,N$, 
an approximate solution of (\ref{eq:1.1}) at $\{t_0,\ldots,t_n\}\times X$, to be determined below, 
and set $v^h(t_k,\cdot)$ by the meshfree interpolation of $\{v_{k,j}\}_{j=1,\ldots,N}$, 
i.e., 
\begin{equation}
\label{eq:2.5.5}
 v^h(t_k,x)=\sum_{j=1}^N\xi_{j}(v_k)\Phi(x,x^{(j)}) 
  + \sum_{\ell=1}^Q\eta_{\ell}(v_k)\pi_{\ell}(x), \quad x\in\Omega, 
\end{equation}
where $v_k = (v_{k,1},\ldots,v_{k,N})^{\mathsf{T}}$. 
Moreover, assume that $v^h$ satisfies (\ref{eq:2.5}) with equality on $X$. 
Then, 
\begin{equation*}
 v_{k+1,j}-v_{k,j}=h\theta\tilde{F}_{k+1,j}(v_{k+1}) 
  +h(1-\theta)\tilde{F}_{k,j}(v_k), \quad k=0,\ldots,n-1, \;\; 
   j=1,\ldots,N.  
\end{equation*}
Here, $\tilde{F}_{k,j}(v_k) = F(t_k,x^{(j)}; v^h(t_k,\cdot))$. 
Thus, denoting 
$\tilde{F}_{k}(v_k)=(\tilde{F}_{k,1}(v_k),\ldots,\tilde{F}_{k,N}(v_k))^{\mathsf{T}}$, 
we get 
\begin{equation}
\label{eq:2.6}
 v_k +h(1-\theta)\tilde{F}_k(v_k) = v_{k+1}-h\theta\tilde{F}_{k+1}(v_{k+1}), 
 \quad k=0,\ldots,n-1.   
\end{equation}
The terminal condition $v^h(t_n,\cdot)$ is given by  
\begin{equation}
\label{eq:2.6.5}
v^h(t_n,x)=I_{f,X}(x), \quad x\in\Omega.
\end{equation}
Consequently, our method is described as follows: determine values of grid points 
$\{t_0,\ldots,t_n\}\times X$ by solving the equation (\ref{eq:2.6}) with (\ref{eq:2.6.5}). 
Then define the function $v^h$ on $\{t_0,\ldots,t_n\}\times \Omega$ by (\ref{eq:2.5.5}), 
which is a candidate of an approximate solution of (\ref{eq:1.1}).

\begin{rem}
\label{rem:2.5}
The linearity of $(\xi(b),\eta(b))$ with respect to $b$ yields 
\begin{equation}
\label{eq:2.7}
 v^h(t_k,x) = v^h(t_{k+1},x) - h(1-\theta)I_{F(t_k,\cdot ; v^h(t_k,\cdot)),X}(x) 
  - h\theta I_{F(t_{k+1},\cdot ; v^h(t_{k+1},\cdot)),X}(x), 
   \quad x\in\Omega. 
\end{equation}
In the case of $\theta=1$, the equation (\ref{eq:2.6}) becomes a simple 
recursion formula, and then $v^h$ is computed by 
the repeated interpolation procedures, i.e., 
\begin{equation*}
 v^h(t_k,x) = v^h(t_{k+1},x) 
  - hI_{F(t_{k+1},\cdot ; v^h(t_{k+1},\cdot)),X}(x), 
   \quad x\in\Omega. 
\end{equation*}
\end{rem}

\section{Convergence}\label{sec:3}

This section is devoted to the proof of convergence of $v^h$ 
constructed in the previous section. 
As stated in the introduction, our main tool is the viscosity solution method in 
\cite{bar-sou:1991}. 
To this end, first we recall the notion of the viscosity solution and 
describe our standing assumptions for \eqref{eq:1.1}. 

An $\mathbb{R}$-valued, upper-semicontinuous function 
$u$ on $[0,T]\times\mathbb{R}^d$ is said to be a viscosity subsolution of 
(\ref{eq:1.1}) if the following two conditions hold: 
\begin{itemize}
 \item[(i)] for every $(t,x)\in [0,T)\times\mathbb{R}^d$ and 
every smooth function $\varphi$ such that $u-\varphi$ has a local 
maximum at $(t,x)$ we have 
\begin{equation*}
 -\partial_t\varphi(t,x)+F(t,x,u(t,x),D\varphi(t,x),
  D^2\varphi(t,x))\le 0; 
\end{equation*}
 \item[(ii)] $u(T,x)\le f(x)$, $x\in\mathbb{R}^d$. 
\end{itemize}
Similarly, an $\mathbb{R}$-valued, lower-semicontinuous function 
$u$ on $[0,T]\times\mathbb{R}^d$ is said to be a viscosity supersolution of 
(\ref{eq:1.1}) if the following two condions hold: 
\begin{itemize}
 \item[(i)] for every $(t,x)\in [0,T)\times\mathbb{R}^d$ and 
every smooth function $\varphi$ such that $u-\varphi$ has a local 
minimum at $(t,x)$ we have 
\begin{equation*}
 -\partial_t\varphi(t,x)+F(t,x,u(t,x),D\varphi(t,x),
   D^2\varphi(t,x))\ge 0;  
\end{equation*}
 \item[(ii)] $u(x)\ge f(x)$, $x\in\mathbb{R}^d$. 
\end{itemize}
We say that $u$ is a viscosity solution of (\ref{eq:1.1}) if 
it is both a viscosity subsolution and a viscosity supersolution 
of (\ref{eq:1.1}). 

We consider the terminal value problem (\ref{eq:1.1}) 
under the following assumptions: 
\begin{assum}
\label{assum:3.1}
\begin{enumerate}
 \item For $t\in [0,T]$, $x\in\mathbb{R}^d$, $z\in\mathbb{R}$, 
$p\in\mathbb{R}^d$, and $\Gamma,\Gamma^{\prime}\in\mathbb{S}^d$ with 
$\Gamma\ge \Gamma^{\prime}$, 
  \begin{equation*}
   F(t,x,z,p,\Gamma)\le F(t,x,z,p,\Gamma^{\prime}).  
  \end{equation*}
\item There exist a continuous function $F_0$ on 
$[0,T]\times\mathbb{R}^d\times\mathbb{R}$ 
and a constant $K_0\in (0,\infty)$ such that 
  \begin{equation*}
   |F(t,x,z,p,\Gamma)- F(t^{\prime},x^{\prime},z^{\prime},p^{\prime},\Gamma)|\le 
   |F_0(t,x,z)-F_0(t^{\prime},x^{\prime},z^{\prime})| 
    + K_0(|p-p^{\prime}|+|\Gamma-\Gamma^{\prime}|)   
  \end{equation*}
for $t,t^{\prime}\in [0,T]$, $x,x^{\prime}\in\mathbb{R}^d$, $z,z^{\prime}\in\mathbb{R}$, 
$p,p^{\prime}\in\mathbb{R}^d$, and
	  $\Gamma,\Gamma^{\prime}\in\mathbb{S}^d$. 
\item There exists a constant $K_1\in (0,\infty)$ such that 
\begin{equation*}
|F(t,x,z,p,\Gamma)|\le K_1(1+|z|+|p|+|\Gamma|)
\end{equation*}
for $t\in [0,T]$, $x\in\mathbb{R}^d$, $z\in\mathbb{R}$, 
$p\in\mathbb{R}^d$, and $\Gamma\in\mathbb{S}^d$. 
\item The function $f$ is continuous and bounded on $\mathbb{R}^d$. 
\end{enumerate}
\end{assum}

We assume that the following comparison principle holds: 
\begin{assum}
\label{assum:3.2}
For every bounded, upper-semicontinuous viscosity subsolution $u$ 
of (\ref{eq:1.1}) 
and bounded lower-semicontinuous viscosity supersolution $w$ of 
(\ref{eq:1.1}), we have 
\begin{equation*}
 u(t,x)\le w(t,x), 
 \quad (t,x)\in [0,T]\times\mathbb{R}^d. 
\end{equation*}
\end{assum}

Under Assumptions \ref{assum:3.1} and \ref{assum:3.2}, there exists a unique continuous 
viscosity solution $v$ of (\ref{eq:1.1}). See \cite{koh-ser:2010}. 

\begin{rem}
It is worth to mention that a wide class of Hamilton-Jacobi-Bellman equations satisfies Assumptions 
\ref{assum:3.1} and \ref{assum:3.2}. Indeed, it can be checked that 
\begin{equation*}
 F(t,x,z,p,\Gamma):=\sup_{a\in A}\left\{-b(t,x,a)^{\mathsf{T}}p-\frac{1}{2}\mathrm{tr}
 ((\sigma\sigma^{\mathsf{T}})(t,x,a)\Gamma\right\} 
\end{equation*}
satisfies Assumption \ref{assum:3.1} provided that 
$A$ is a compact subset of some Euclidean space, the functions 
$b$ and $\sigma$ are bounded and Lipschitz continuous in $x$. 
In this case, it is known that Assumption \ref{assum:3.2} is also satisfied. 
See Theorem 9.1 in Fleming and Soner \cite{fle-son:2006}. 
\end{rem}

\begin{assum}
\label{assum:3.3}
The equation (\ref{eq:2.6}) has a unique solution $v_{k,j}$, 
$k=0,\ldots,n-1$, $j=1,\ldots,N$. 
\end{assum}
Notice that Assumption \ref{assum:3.3} is trivially satisfied when $\theta=1$. 

Set $L_N=|A_{\Phi,X}^{-1}|$ if $\Phi$ is simply positive definite and 
$L_N=|\tilde{A}_{\Phi,X}^{-1}|$ if $\Phi$ is conditionally positive definite of 
order $m\ge 1$, where $\tilde{A}_{\Phi,X}$ denotes the matrix on the 
left-hand side in (\ref{eq:2.3}). 

Hereafter, we assume that 
the number $N$ of data sites is a function of the time step $h$. 
To control the bound of $v^h$, we make the following assumptions: 
\begin{assum}
\label{assum:3.4}
The function $L_N$ of $h$ is bounded away from zero and 
there exists $K_3\in (0,\infty)$, $\delta\in (0,1/5)$, 
$h_0\in (0,1)$ such that 
\begin{equation*}
 h^{\delta}\sqrt{N}L_N
 \exp(\sqrt{3}TK_1K_2(1+\sqrt{N})L_N)\le K_3, \quad h\le h_0,  
\end{equation*}
where 
$$
 K_2= \max\left\{\left(\sum_{|\alpha|\le 3}
 \max_{x,y\in\Omega}|D^{\alpha}\Phi(x,y)|^2\right)^{1/2}, 
  \left(\sum_{|\alpha|\le 3}
 \max_{x\in\Omega}\sum_{\ell=1}^Q|D^{\alpha}\pi_{\ell}(x)|^2\right)^{1/2}
 \right\}. 
$$
Here, $D^{\alpha}\Phi(x,y)$ is interpreted as the partial derivative of 
$\Phi$ with respect to the first argument.  
\end{assum}
To discuss Assumption \ref{assum:3.4}, recall that 
the set $X$ of data sites is said to be quasi-uniform with respect to 
a constant $c_{qu}>0$ if 
$$
  q_X\le \Delta_{\Omega,X}\le c_{qu}q_X,  
$$
where $q_X$ is the separation distance of $X$, defined by 
$$
 q_X=\frac{1}{2}\min_{i\neq j}|x^{(i)}-x^{(j)}|. 
$$
A typical example of quasi-uniform data sites is, of course, a set of 
uniformly spaced grid points. 
It is known that if $X$ is quasi-uniform with respect to $c_{qu}>0$,  
then there exists constants $c_1, c_2>0$, only depending on 
$d$ and $c_{qu}$, such that 
$$
  c_1N^{-1/d}\le q_X\le c_2N^{-1/d}. 
$$
\begin{ex}
\begin{enumerate}
\item In the case of $\Phi(x,y)=e^{-\alpha |x-y|^2}$, $\alpha>0$, 
it is known that 
$$
 |A_{\Phi,X}^{-1}|\le \frac{(2\alpha)^{d/2}}{\tilde{c}_{d,1}}q_X^d
  e^{40.71d^2/(\alpha q_X^2)} 
$$
where
\begin{equation*}
 \tilde{c}_{d,1}=\frac{1}{2\Gamma((d+2)/2)}\left(\frac{\tilde{c}_{d,2}}{\sqrt{8}}\right)^d, 
 \quad \tilde{c}_{d,2}=12\left(\frac{\pi\Gamma^2((d+2)/2)}{9}\right)^{1/(d+1)},  
\end{equation*}
and $\Gamma$ denotes the Gamma function (see \cite[Chapter 12]{wen:2010}). 
Thus, if $X$ is quasi-uniform, then 
$$
 L_N=|A_{\Phi,X}^{-1}|\le \frac{(2\alpha)^{d/2}}{\tilde{c}_{d,1}}c_2^2N^{-1}
  e^{40.71d^2N^{2/d}/(\alpha c_1^2)}.
$$
\item In the case of $\Phi(x,y)=(\alpha^2+|x-y|^2)^{-\beta}$, 
$\alpha, \beta>0$, it is known that 
$$
 |A_{\Phi,X}^{-1}|\le \tilde{c}_{d,\alpha,\beta} q_X^{\beta+d/2-1/2}
  \exp(2\alpha\tilde{c}_{d,2}/q_X) 
$$
with an explicitly known constant $\tilde{c}_{d,\alpha,\beta}$ 
(see \cite[Chapter 12]{wen:2010}). 
Thus, if $X$ is quasi-uniform, then 
$$
 L_N=|A_{\Phi,X}^{-1}|\le \tilde{c}_{d,\alpha,\beta} c_2^{\beta + d/2-1/2}
 N^{-(\beta+d/2-1/2)/d}e^{2\alpha\tilde{c}_{d,2}N^{1/d}/c_1}.
$$
\end{enumerate}
\end{ex}

To ensure the convergence of the interpolation at each time step, 
we impose the following conditions in view of (\ref{eq:2.4}): 
\begin{assum}
\label{assum:3.6}
\begin{enumerate}
\item The terminal data $f$ and the function $F(t_k,\cdot ; v^h(t_k,\cdot))$ belong to 
$\mathcal{N}_{\Phi}(\Omega)$ for every $k=0,\ldots,n-1$.  
\item The meshfree approximation at each time step is successful, i.e.,      
\begin{equation*}
 \Delta_{X,\Omega}^{\nu}\left(1+\max_{k=0,\ldots,n-1} 
  |F(t_k,\cdot ; v^h(t_k,\cdot))|_{\mathcal{N}_{\Phi}(\Omega)}\right) 
  \to 0, \quad h\to 0. 
\end{equation*}
\end{enumerate}
\end{assum}

To prove the convergence, we define $v^h(s,x)$ for $s\in (t_k, t_{k+1})$ 
by any continuous interpolation of 
$v^h(t_k,x)$ and $v^h(t_{k+1},x)$, $k=0,\ldots,n-1$. 
\begin{thm}
\label{thm:3.7}
Suppose that Assumptions \ref{assum:3.1}-\ref{assum:3.4}, \ref{assum:3.6} hold. 
Then we have   
\begin{equation*}
 \lim_{h\searrow 0, \; s\to t}\sup_{x\in\Omega}|v^{h}(s,x)-v(t,x)|=0.  
\end{equation*}
\end{thm}

\begin{rem}
\label{rem:3.8}
In the theorem above, 
in addition to Assumptions \ref{assum:3.1}-\ref{assum:3.4} and \ref{assum:3.6}, 
we have assumed $\Omega$ to be a bounded open subset of $\mathbb{R}^d$ 
and to satisfy an interior cone condition, and 
$X$ to be $\Pi_{m-1}(\mathbb{R}^d)$-unisolvent. 
All these conditions are satisfied when 
$\Omega$ is a star-shaped set and $X$ is a set of uniformly spaced grid 
points in $\Omega$ with $N\ge m$. 
\end{rem}

Notice that our approximation method is described in the form 
$v^h(t_k,x)=G^h(t_k, x, v^h(t_{k+1}))$. The arguments in \cite{bar-sou:1991} tell us that 
if the function $v^h$ is bounded with respect to $h$, 
the operator $G^h$ is monotone with respect to the last argument, 
and $G^h$ has a consistency property related to the equation \eqref{eq:1.1}, 
then we can basically show its convergence, i.e., Theorem \ref{thm:3.7}.  
In our situation, however, the monotonicity property is nontrivial since 
$G^h$ contains the derivative terms. 
We will overcome this difficulty by proving a variant of Lemma 4.1 in 
\cite{koh-ser:2010} (Lemma \ref{lem:3.10} below) in our case. 
This lemma means that $G^h$ has the monotonicity property with negligible term 
as well as the consistency property. Moreover, by an argument similar to that in the 
proof of this lemma, we can show the boundedness of $v^h$ (Lemma \ref{lem:3.12}).

We start with two lemmas (Lemmas \ref{lem:3.8} and \ref{lem:3.9}) to present 
estimation results for $v^h$. 
\begin{lem}
\label{lem:3.8}
Suppose that Assumptions \ref{assum:3.1}-\ref{assum:3.4} hold. 
Then, there exist $h_1\in (0,1)$ such that 
for $h\le h_1$  
\begin{equation*}
 \max_{k=0,\ldots,n-1}|v_{k}|
 \le \left(\sup_{x\in\Omega}|f(x)| + \frac{1}{\sqrt{2}K_2}\right)
  \exp(\sqrt{3}TK_1K_2\sqrt{N}(1+\sqrt{N})L_N). 
\end{equation*}
\end{lem}
\begin{proof} 
Since $|\xi(v_k)|+|\eta(v_k)|\le \sqrt{2}L_N|v_k|$ for any $k$, we have    
\begin{equation*}
 \sum_{i=0}^2|D^iv^h(t_k,x)| 
 \le \sqrt{2}K_2(1+\sqrt{N})L_N|v_k|. 
\end{equation*}
This and Assumption \ref{assum:3.1} imply 
\begin{equation*}
|\tilde{F}_{k,j}(v_k)|\le K_1+K_1\sum_{i=0}^2\left|
 D^{i}v^h(t_k,x^{(j)})\right|
 \le K_1+\sqrt{2}K_1K_2(1+\sqrt{N})L_N|v_k|.   
\end{equation*}
Using $|y|\le \sqrt{N}\max_{j=1,\ldots,N}|y_j|$ for 
$y=(y_1,\ldots,y_N)^{\mathsf{T}}\in\mathbb{R}^N$, 
we find that 
\begin{equation*}
 |\tilde{F}_k(v_k)|\le K_1\sqrt{N} + \sqrt{2}K_1K_2\sqrt{N}L_N|v_k| 
  + \sqrt{2}K_1K_2NL_N|v_k|. 
\end{equation*}
Hence, 
\begin{align*}
 |v_k|&\le |v_{k+1}| + h(1-\theta)|\tilde{F}_k(v_k)| 
  +h\theta|\tilde{F}_{k+1}(v_{k+1})| \\
  &\le \left(1+\sqrt{2}h\theta K_1K_2\sqrt{N}(1+\sqrt{N})L_N\right)|v_{k+1}| 
   + \sqrt{2}h(1-\theta)K_1K_2\sqrt{N}(1+\sqrt{N})L_N|v_k| +hK_1\sqrt{N}. 
\end{align*}
Assumption \ref{assum:3.4} implies 
$$
 h^{\delta}\sqrt{N}L_N(1+\sqrt{N})\le h^{\delta}\sqrt{N}L_N\frac{L_N}{C_0}(1+\sqrt{N}) 
 \le h^{\delta}\sqrt{N}L_N\frac{\exp(\sqrt{3}TK_1K_2(1+\sqrt{N})L_N)}{\sqrt{3}C_0TK_1K_2}
 \le C 
$$
for $h\le h_0$, where $C_0$ is a lower bound for $L_N$ . Thus, 
$$
 \sqrt{2}h(1-\theta)K_1K_2\sqrt{N}L_N(1+\sqrt{N})\le 1-\frac{\sqrt{2}}{\sqrt{3}}<1, 
 \quad h\le h_1 
$$
for some $h_1\le h_0$,  
it follows that for any $k=0,\ldots,n-1$, $h\le h_1$,  
\begin{align*}
|v_k|&\le \frac{1+\sqrt{2}h\theta K_1K_2\sqrt{N}(1+\sqrt{N})L_N}
 {1-\sqrt{2}h(1-\theta) K_1K_2\sqrt{N}(1+\sqrt{N})L_N}|v_{k+1}| 
  + \frac{hK_1\sqrt{N}}
  {1-\sqrt{2}h(1-\theta) K_1K_2\sqrt{N}(1+\sqrt{N})L_N} \\ 
  &= \left(1+\frac{\sqrt{2}h K_1K_2\sqrt{N}(1+\sqrt{N})L_N}
 {1-\sqrt{2}h(1-\theta) K_1K_2\sqrt{N}(1+\sqrt{N})L_N}\right)|v_{k+1}| 
  + \frac{hK_1\sqrt{N}}
  {1-\sqrt{2}h(1-\theta) K_1K_2\sqrt{N}(1+\sqrt{N})L_N} \\
 &\le (1+\sqrt{3}hK_1K_2\sqrt{N}(1+\sqrt{N}L_N))|v_{k+1}| 
  + \sqrt{3/2}hK_1\sqrt{N}.   
\end{align*}
Therefore, we have, for any $k$ 
\begin{align*}
 |v_k|&\le (1+\sqrt{3}hK_1K_2\sqrt{N}(1+\sqrt{N})L_N)^n\sup_{x\in\Omega}|f(x)| \\ 
  &\quad + \sqrt{3/2}hK_1\sqrt{N}\times 
  \frac{(1+\sqrt{3}hK_1K_2\sqrt{N}(1+\sqrt{N})L_N)^n -1}
   {\sqrt{3}hK_1K_2\sqrt{N}(1+\sqrt{N})L_N} \\  
  &\le \exp(\sqrt{3}TK_1K_2\sqrt{N}(1+\sqrt{N})L_N)\sup_{x\in\Omega}|f(x)| \\
   &\quad + \frac{1}{\sqrt{2}K_2(1+\sqrt{N})L_N}
     (\exp(\sqrt{3}TK_1K_2\sqrt{N}(1+\sqrt{N})L_N)-1),  
\end{align*}
leading to the conclusion of the lemma. 
\end{proof}

\begin{lem}
\label{lem:3.9}
Suppose that Assumptions $\ref{assum:3.1}$, $\ref{assum:3.3}$ 
and $\ref{assum:3.4}$ hold. 
Then there exist a constant $K_4\in (0,\infty)$, $h_2\in (0,1]$ such that 
for $h\le h_2$ we have the following: 
\begin{enumerate}[\rm (i)]
 \item $\sum_{|\alpha|_1\le 3}|D^{\alpha}v^h(t_k,x)|\le K_4h^{-\delta}$ 
  for $k=0,\ldots,n-1$ and $x\in\Omega$. 
 \item $\sum_{|\alpha|_1\le 3}|D^{\alpha}v^h(t_{k+1},x)-D^{\alpha}v^h(t_k,x)|\le K_4h^{1-2\delta}$ 
  for $k=0,\ldots,n-2$ and $x\in\Omega$. 
\end{enumerate}
\begin{proof}
Fix $k=0,\ldots,n-1$ and 
let $h_1$ be as in Assumption $\ref{assum:3.4}$.  
Using the previous lemma, we observe 
\begin{align*}
 \sum_{|\alpha|_1\le 3}|D^{\alpha}v^h(t_k,x)| 
 &\le \sqrt{2}K_2(1+\sqrt{N})L_N|v_k| \\ 
 &\le 2\left(1+\sqrt{2}K_2\sup_{x\in\Omega}|f(x)|\right)
 \sqrt{N}L_N\exp(\sqrt{3}TK_1K_2\sqrt{N}(1+\sqrt{N})L_N) \\
 &\le 2\left(1+\sqrt{2}K_2\sup_{x\in\Omega}|f(x)|\right)
  K_3h^{-\delta} 
\end{align*}
for $h\le h_1$. Thus the first assertion follows. 

Next, since $\xi(b)$ and $\eta(b)$ is linear in $b$, 
we obtain  
\begin{align*}
 \sum_{|\alpha|_1\le 3}|D^{\alpha}v^h(t_{k+1},x) -D^{\alpha}v^h(t_k,x)| 
 &\le \sqrt{2}K_2\sqrt{N}|\xi(v_{k+1})-\xi(v_k)| + \sqrt{2}K_2|\eta(v_{k+1})-\eta(v_k)| \\
 &\le \sqrt{2}K_2(1+\sqrt{N})L_N|v_{k+1}-v_k|. 
\end{align*}
Using Assumption \ref{assum:3.1} and the first assertion in this lemma, we see 
\begin{align*}
 |\tilde{F}_k(v_k)|&\le\left(\sum_{j=1}^NK_1^2\left(1+\sum_{i=0}^2
  |D^iv^h(t_k,x^{(j)})|\right)^2\right)^{1/2} 
  \le\sqrt{N}K_1\left(1+\sum_{i=0}^2\sup_{x\in\Omega}|D^iv^h(t_k,x)|\right) \\
  &\le C\sqrt{N}h^{-\delta}  
\end{align*}
for $h\le h_1$. Hence, 
\begin{equation*}
|v_{k+1}-v_k|\le h|\tilde{F}_k(v_k)| + h|\tilde{F}_{k+1}(v_{k+1})|
\le C\sqrt{N}h^{1-\delta}. 
\end{equation*}
Therefore, in view of Assumption \ref{assum:3.4},  
\begin{equation*}
 \sum_{|\alpha|_1\le 3}|D^{\alpha}v^h(t_{k+1},x) -D^{\alpha}v^h(t_k,x)|
 \le C(1+\sqrt{N})\sqrt{N}L_Nh^{1-\delta}\le Ch^{1-2\delta} 
\end{equation*}
for sufficiently small $h$. Thus the second assertion follows. 
\end{proof}
\end{lem}

Let $K_4$ as in the previous lemma. 
For $h>0$ and $\kappa>0$ define 
\begin{equation*}
 \mathcal{D}_{h,\delta}=\left\{(p,\Gamma)\in\mathbb{R}^d\times 
  \mathbb{S}^d: |p|, |\Gamma|\le K_4h^{-\delta}\right\}, \quad 
 \mathcal{X}_{h,\kappa}=\left\{w\in\mathbb{R}^d: |w|\le h^{-\kappa}\right\}. 
\end{equation*}
The following lemma is a variant of Lemma 4.1 in \cite{koh-ser:2010}.  
\begin{lem}
\label{lem:3.10}
Suppose that Assumption $\ref{assum:3.1}$ holds. 
Let $O\subset\mathbb{R}^d$ be open and bounded. Then 
for any open ball $U$ compactly included in $O$ 
there exist $h_3\in (0,1]$, $\beta\in (0,\infty)$, 
$\kappa\in (0,\infty)$ such that 
for $(t,x,z)\in [0,T]\times U\times \mathbb{R}$,  
$\{\varphi^h\}_{h\in (0,h_3]}\subset C^3(O)$ with 
$\sum_{|\alpha|_1\le 3}\sup_{y\in O}|D^{\alpha}\varphi^h(y)|
\le K_4h^{-\delta}$, and $h\in (0,h_3]$, 
\begin{align*}
 &\bigg|\varphi^h(x)-hF(t,x,z,D\varphi^h(x),D^2\varphi^h(x)) \\
 & \quad -\sup_{(p,\Gamma)\in\mathcal{D}_{h,\delta}}\inf_{w\in\mathcal{X}_{h,\kappa}}
  \left[\varphi^h(x+\sqrt{h}w) 
 -\sqrt{h}w^{\mathsf{T}}p-\frac{h}{2}w^{\mathsf{T}}\Gamma w  
 -hF(t,x,z,p,\Gamma)\right]\bigg| \le C_{K_0,K_4}h^{1+\beta}.   
\end{align*}
\end{lem}
\begin{proof}
First, fix arbitrary $\varphi^h\in C^3(O)$ with 
$\sum_{|\alpha|_1\le 3}\sup_{y\in O}|D^{\alpha}\varphi^h(y)|
\le K_4h^{-\delta}$ 
and $(t,x,z)\in [0,T]\times U\times\mathbb{R}$. 
Then set $\varphi=\varphi^h$ and $p_0=D\varphi(x)$, $\Gamma_0=D^2\varphi(x)$. 
Also, for simplicity, we write $F(p,\Gamma)=F(t,x,z,p,\Gamma)$ for 
$(p,\Gamma)\in\mathcal{D}_{h,\delta}$. 
Since $\delta<1/5$, there exists $\varepsilon>0$ such that 
$\delta<1/(5+\varepsilon)$. Then define $\kappa>0$ by 
$$
  \kappa = \frac{1}{3}\left(\frac{5}{10+2\varepsilon}-\delta\right). 
$$

Next, take $h_3\in (0,1]$ such that $x+\sqrt{h}w\in O$ for all $x\in U$,  
$w\in\mathcal{X}_{h,\kappa}$, and $h\in (0,h_3]$. 
By Taylor expansion of $\varphi$ up to the second term, we have 
\begin{align*}
 &\sup_{(p,\Gamma)\in\mathcal{D}_{h,\delta}}\inf_{w\in\mathcal{X}_{h,\kappa}}
  \left[\varphi(x+\sqrt{h}w) 
 -\sqrt{h}w^{\mathsf{T}}p-\frac{h}{2}w^{\mathsf{T}}\Gamma w  
 -hF(p,\Gamma)\right] \\
 &\ge \varphi(x)-Ch^{-\delta+3/2-3\kappa} 
  + \sup_{(p,\Gamma)\in\mathcal{D}_{h,\delta}}\inf_{w\in\mathcal{X}_{h,\kappa}}
  \left[\sqrt{h}w^{\mathsf{T}}(p_0-p) 
   +\frac{h}{2}w^{\mathsf{T}}(\Gamma_0-\Gamma) w  
 -hF(p,\Gamma)\right]. 
\end{align*}
Then, considering $p=p_0$ and $\Gamma=\Gamma_0$, we find that 
the right-hand side in the above inequality is greater than 
$\varphi(x)-Ch^{1+\varepsilon/(10+2\varepsilon)} -hF(p_0,\Gamma_0)$. 

To show the reverse inequality, let $(p,\Gamma)\in\mathcal{D}_{h,\delta}$. 
Since $2\kappa-\delta=(5/3)(1/(5+\varepsilon)-\delta)>0$, we can take 
$\gamma\in (0,2\kappa-\delta)$. 
Suppose that the minimum eigenvalue of $\Gamma_0-\Gamma$ is
greater than or equal to $-h^{\gamma}$. Then $\Gamma\le \Gamma_0+h^{\gamma}I$ 
so that 
$$
F(p,\Gamma)\ge F(p,\Gamma_0+h^{\gamma}I)\ge F(p_0,\Gamma_0)
 -K_0|p-p_0|-K_0h^{\gamma}. 
$$
The last inequality follows from Assumption \ref{assum:3.1} (ii), 
i.e., the Lipschitz continuity of $F(p,\Gamma)$. 
Thus, 
\begin{equation}
\label{eq:3.1}
\begin{split}
 &\sqrt{h}(p_0-p)^{\mathsf{T}}w+\frac{h}{2}w^{\mathsf{T}}(\Gamma_0-\Gamma)w
  -hF(p,\Gamma) \\
 &\le \sqrt{h}(p_0-p)^{\mathsf{T}}w+K_4h^{1-\delta}|w|^2-hF(p_0,\Gamma_0) 
  +K_0h|p-p_0| + K_0h^{1+\gamma}.  
\end{split}
\end{equation}
In case $p=p_0$ we take $w=0$ so that the right-hand side in
 (\ref{eq:3.1}) becomes $-hF(p_0,\Gamma_0)+K_0h^{1+\gamma}$. 
Otherwise, by the choice $w=-h^{\delta}(p_0-p)/|p_0-p|$,  
the right-hand side in (\ref{eq:3.1}) becomes 
\begin{align*}
 &-h^{1/2+\delta}|p_0-p|+K_4h^{1+\delta}
  -hF(p_0,\Gamma_0) +K_0h|p_0-p|+K_0h^{1+\gamma} \\ 
 &=|p_0-p|(-h^{1/2+\delta}+K_0h) -hF(p_0,\Gamma_0)+K_4h^{1+\delta}
   +K_0h^{1+\gamma} \\
 &\le -hF(p_0,\Gamma_0)+(K_0+K_4)h^{1+\min{\{\gamma,\delta\}}}
\end{align*}
for any sufficiently small $h$ since there exists $h_3^{\prime}\in (0,h_3]$ such that 
$-h^{1/2+\delta}+K_0h\le 0$ for all $h\in (0,h_3^{\prime}]$. 

Suppose that the minimum eigenvalue $\mu$ of $\Gamma_0-\Gamma$ 
is less than $-h^{\gamma}$. Then take $w\neq 0$ as an eigenvector 
with respect to $\mu$ such that $(p_0-p)^{\mathsf{T}}w\le 0$ 
and $|w|=h^{-\kappa}$. 
This choice yields 
\begin{align*}
 &\sqrt{h}(p_0-p)^{\mathsf{T}}w +
 \frac{h}{2}w^{\mathsf{T}}(\Gamma_0-\Gamma)w -hF(p,\Gamma) \\
 &\le \frac{h}{2}\mu |w|^2 -hF(p_0,\Gamma_0) +hK_0
  (|p-p_0|+|\Gamma-\Gamma_0|) \\
 &\le -\frac{h}{2}h^{\gamma}h^{-2\kappa}-hF(p_0,\Gamma_0)+4K_0K_4h^{1-\delta} 
 \le -hF(p_0,\Gamma_0)+\frac{h^{-\delta}}{2}
  (-h^{1+\gamma-2\kappa+\delta}+8K_0K_4h), 
\end{align*}
and the right-hand side in the last inequality just above 
is at most $-hF(p_0,\Gamma_0)$ for any sufficiently small $h$ 
since there exists $h_3^{\prime\prime}\in (0,h_3^{\prime}]$ such that 
$-h^{1+\gamma-2\kappa+\delta}+8K_0k_4h\le 0$ for all 
$h\in (0,h_3^{\prime\prime}]$. 

Therefore, we have proved that for any $(p,\Gamma)\in\mathcal{D}_{h,\delta}$, 
\begin{equation*}
 \inf_{w\in\mathcal{X}_h}\left[\sqrt{h}(p_0-p)^{\mathsf{T}}w+
  \frac{h}{2}w^{\mathsf{T}}(\Gamma_0-\Gamma)w-hF(p,\Gamma)\right]
 \le -hF(p_0,\Gamma_0) +Ch^{1+\beta} 
\end{equation*}
for some $\beta=\beta_{\delta}$. 
Combining this with Taylor expansion of $\varphi$ up to the second term, 
we obtain 
\begin{equation*}
 \sup_{(p,\Gamma)\in\mathcal{D}_{h,\delta}}\inf_{w\in\mathcal{X}_{h,\kappa}}
  \left[\varphi(x+\sqrt{h}w)-\sqrt{h}p^{\mathsf{T}}w 
  -\frac{h}{2}w^{\mathsf{T}}\Gamma w-hF(p,\Gamma)\right]
 \le -hF(p_0,\Gamma_0) +Ch^{1+\beta}, 
\end{equation*}
which completes the proof of the lemma. 
\end{proof}

The function $v^h$ is actually locally bounded with respect to $x$ uniformly in $h$. 
\begin{lem}
\label{lem:3.12}
Under the assumptions imposed in Theorem \ref{thm:3.7}, 
for any open ball $O$ with $\overline{O}\subset\Omega$ 
there exist $h_4\in (0,1]$ such that 
\[
 \sup_{h\in (0,h_4]}\max_{k=0,\ldots,n}\sup_{x\in O}|v^h(t_k,x)|<\infty. 
\]
\end{lem}
\begin{proof}
Assumption \ref{assum:3.6} and $f\in C(\Omega)$ mean 
$$
|v^h(t_n,x)|\le B_n, \quad x\in O, \;\; h\in (0,1]
$$
for some positive constant $B_n$.  
So suppose that for $k\le n-1$ there exists $B_{k+1}>0$ such that 
$$
 |v^h(t_{k+1},x)|\le B_{k+1}, \quad x\in O, \;\; h\in (0,h_4]
$$
with some $h_4\in (0,1]$ to be determined below.

By a routine argument, we can extend each $v^h(t_i,\cdot)$ to a $C^3$-function on $\Omega$. Indeed, 
recall that there exists a $C^{\infty}$-function $\zeta$ compactly supported in $\Omega$ such that 
$0\le\zeta\le 1$ on $\Omega$ and $\zeta=1$ on $O$. See, e.g., 
Theorem 1.4.1 in H{\"o}rmander \cite{hor:1990}. 
Then, for $i=0,1,\ldots,n-1$ define $u^h(t_i,\cdot)\in C^3(\Omega)$ by 
\[
 u^h(t_i,x) = v^h(t_i,x)\zeta(x) + K_4h^{-\delta}(1-\zeta(x)), \quad x\in\Omega. 
\]
It is straightforward to see that  
$D^{\alpha}u^h(t_i,x)=D^{\alpha}v^h(t_i,x)$ for $x\in O$ and 
\[
 \sum_{|\alpha|_1\le 3}\sup_{x\in\mathbb{R}^d}|D^{\alpha}u^h(t_i,x)|\le K_5h^{-\delta}, \quad x\in\Omega, 
\]
for some positive constant $K_5$. By abuse of notation we still write $v^h$ for $u^h$.

To get a bound of $v^h(t_k,\cdot)$, rewrite $v^h(t_{k},x)$ as 
\begin{equation*}
\label{eq:3.3}
 v^{h}(t_k,x) 
  = v^{h}(t_{k+1},x)-hF(t_{k+1},x;v^{h}(t_{k+1},\cdot)) 
  +hR^{h}_1(x) + hR_2^h(x) + h R^h_3(x), 
\end{equation*}
where  
\begin{gather*}
 R^{h}_1(x)=(1-\theta)\left(F(t_k,x;v^{h}(t_k,\cdot)) - 
   I_{F(t_k,\cdot;v^{h}(t_k,\cdot)),X}(x)\right), \\ 
  R_2^h(x) = \theta\left(F(t_{k+1},x;v^{h}(t_{k+1},\cdot)) - 
    I_{F(t_{k+1},\cdot;v^{h}(t_{k+1},\cdot)),X}(x)\right), \\
  R_3^h(x) = (1-\theta)\left(F(t_{k+1},x;v^{h}(t_{k+1},\cdot)) 
    - F(t_k,x;v^{h}(t_k,\cdot))\right). 
\end{gather*}
Further, note that by Assumption \ref{assum:3.1}, 
\begin{equation}
\label{eq:3.1.5}
\begin{split}
&|F(t_{k+1},x;v^{h}(t_{k+1},\cdot)) - F(t_k,x;v^{h}(x,\cdot))|  \\
&\le |F_0(t_{k+1},x,v^{h}(t_{k+1},x)) - F_0(t_k,x,v^h(t_k,x))| \\
&\quad + K_0|Dv^{h}(t_{k+1},x)-Dv^{h}(t_k,x)| 
  +K_0|D^2v^{h}(t_{k+1},x)-D^2v^{h}(t_k,x)|.
\end{split}  
\end{equation}
Assumption \ref{assum:3.6}, Lemma \ref{lem:3.9} 
and (\ref{eq:3.1.5}) then guarantee that 
$\sum_{i=1}^3\sup_{x\in\Omega}|R^h_i(x)|$ is bounded with respect to $h$. 
Thus Lemma \ref{lem:3.10} yields  
$|v^h(t_k,x)|\le |Q| + Ch$ where $x\in O$ and 
\begin{equation*}
Q=\sup_{(p,\Gamma)\in\mathcal{D}_{h,\delta}}\inf_{w\in\mathcal{X}_{h,\kappa}}
  \left[\tilde{u}^h(t_{k+1},x+\sqrt{h}w) 
 -\sqrt{h}w^{\mathsf{T}}p-\frac{h}{2}w^{\mathsf{T}}\Gamma w  
 -hF(t,x,v^h(t_{k+1},x),p,\Gamma)\right]. 
\end{equation*}
Considering $p=0$ and $\Gamma=0$, we see 
$Q\ge -(1+K_1h)B_{k+1}-K_1h$. 

To obtain an upper bound, observe 
$$
 Q\le B_{k+1}+  
  \sup_{(p,\Gamma)\in\mathcal{D}_{h,\delta}}\inf_{w\in\mathcal{X}_{h,\kappa}}
   Q_{p,\Gamma ,w}, 
$$ 
where 
$$
Q_{p,\Gamma,w}= -\sqrt{h}w^{\mathsf{T}}p-\frac{h}{2}w^{\mathsf{T}}\Gamma w 
    -hF(t,x,v^h(t_{k+1},x),p,\Gamma).
$$ 
Then we will show that for any $(p,\Gamma)\in\mathcal{F}_{h,\delta}$ 
we can find $w\in\mathcal{X}_{h,\kappa}$ satisfying 
$Q_{p,\Gamma,w}\le K_1hB_{k+1} + Ch$. 
So fix $(p,\Gamma)\in\mathcal{F}_{h,\delta}$. 
First assume that the minimum eigenvalue of $-\Gamma$ is 
greater than or equal to $-h^{\gamma}$. 
If $p=0$ then we may take $w=0$, leading to 
$Q_{p,\Gamma,w}\le -hF(t,x,v^h(t_{k+1},x),0,h^{\gamma}I) 
\le K_1h+K_1hB_{k+1}+K_1h^{1+\gamma}$. 
Otherwise, take $w=h^{\delta}p/|p|$. Then we see 
\begin{align*}
Q_{p,\Gamma,w}&\le -h^{(1/2)+\delta}|p|+\frac{h^{1+\delta}}{2} + 
 K_1h(1+B_{k+1}) + K_1h|p| + K_1h^{1+\gamma} \\
 &\le |p|(-h^{(1/2)+\delta} + K_1h) +Ch+ K_1hB_{k+1} \le Ch+K_1hB_{k+1}
\end{align*}
since $-h^{(1/2)+\delta}+K_1h\le 0$ for $h\in (0, h_4^{\prime}]$ 
with some $h_4^{\prime}\in (0,h_3]$. 

Next assume that the minimum eigenvalue of $-\Gamma$ is less than 
$-h^{\gamma}$. Then take $w$ to be the corresponding eigenvector satisfying 
$-p^{\mathsf{T}}w\le 0$ and $|w|=h^{-\kappa}$. 
This choice leads to 
\begin{equation*}
 Q_{p,\Gamma,w}\le -\frac{h^{1+\gamma-2\kappa}}{2}+ K_1h(1+B_{k+1}) 
  +2K_1h^{1-\delta} \le K_1h(1+B_{k+1})
\end{equation*}
since there exists $h_4\in (0,h_4^{\prime}]$ such that 
$-h^{1+\gamma-2\kappa +\delta} + 4K_1h\le 0$ for 
$h\in (0,h_4]$. 

Therefore we deduce that $|Q|\le (1+K_1h)B_{k+1}+Ch$ for $h\le h_4$. 
Denoting the right-hand side by $B_k$, we obtain the sequence $\{B_k\}$ 
satisfying $B_k=(1+K_1h)B_{k+1}+Ch$. By a routine argument we have 
$B_k\le e^{TK_1}B_n + Ce^{TK_1}$ for all $k$. 
Thus the lemma follows. 
\end{proof}

\begin{proof}[Proof of Theorem \ref{thm:3.7}]
We adopt the viscosity solution method as stated in
 \cite{bar-sou:1991}. To this end, we set 
 $v^h(t,x)=v(t,x)$ for $(t,x)\in [0,T]\times (\mathbb{R}^d\setminus\Omega)$ 
 and consider 
\begin{equation*}
 \overline{v}(t,x)=\limsup_{{s\to t, \; y\to x}
  \atop{h\searrow 0}} v^{h}(s,y), 
 \quad (t,x)\in [0,T]\times\mathbb{R}^d   
\end{equation*}
to show that $\overline{v}$ is a viscosity subsolution of
 (\ref{eq:1.1}). 
 Lemma \ref{lem:3.12} implies that $\overline{v}$ is finite on 
 $[0,T]\times\mathbb{R}^d$. 

Fix $(t,x)\in [0,T)\times\Omega$ and 
let $\varphi\in C^3([0,T]\times\mathbb{R}^d)$ such that 
$\overline{v}-\varphi$ has a local maximum at $(t,x)$. 
Then, take $r>0$ such that 
$$
 (\overline{v}-\varphi)(s,y)\le (\overline{v}-\varphi)(t,x), \quad (s,y)\in B_r(t,x), 
$$
where $B_r(t,x)$ is the closed ball centered at $(t,x)$ with radius $r$, 
and that $B_r(t,x)\subset [0,T]\times\Omega$. 
Next, for $(s,y)\in B_r(t,x)$ set 
$$
 \tilde{\varphi}(s,y)=\varphi(s,y)-(\varphi(t,x)-\overline{v}(t,x)) 
 + |s-t|^2+|y-x|^2. 
$$
It follows that $\overline{v}(t,x)=\tilde{\varphi}(t,x)$ and that 
$(t,x)$ is a strict maximum of 
$\overline{v}-\tilde{\varphi}$ on $B_r(t,x)$. 
By abuse of notation, we write $\varphi$ for 
$\tilde{\varphi}$. 

By definition of $\overline{v}$, there exist $h_m$ and 
$(\tilde{s}_m,\tilde{y}_m)\in B_r(t,x)$ 
such that, as $m\to \infty$, 
$$
 h_m\to 0, \;\; (\tilde{s}_m,\tilde{y}_m)\to (t,x), \;\; 
 v^{h_m}(\tilde{s}_m,\tilde{y}_m)\to \overline{v}(t,x). 
$$
Take $s_m$ and $y_m$ so that $s_m=ih_m$ for some 
$i=i_m=0,\ldots,n-1$ and that 
\begin{equation}
\label{eq:3.2}
  (v^{h_m}-\varphi)(s_m,y_m)\ge \sup_{(s,y)\in B_r(t,x)}
 (v^{h_m}-\varphi)(s,y)-h_m^{3/2}. 
\end{equation} 
Moreover, the sequence $(s_m,y_m)$, $m\ge 1$, can be taken from the bounded set 
$B_r(t,x)$, so there exists a limit point 
$(\tilde{t},\tilde{x})\in B_r(t,x)$ possibly along a subsequence. 
Thus, denoting $c_m=(v^{h_m}-\varphi)(s_m,y_m)$, we have 
\begin{equation*}
 0=(\overline{v}-\varphi)(t,x)=\lim_{m\to\infty}(v^{h_m}-\varphi)
 (\tilde{s}_m,\tilde{y}_m)
  \le \liminf_{m\to\infty}c_m\le\limsup_{m\to\infty}c_m \\
  \le (\bar{v}-\varphi)(\tilde{t},\tilde{x}). 
\end{equation*}
Since $(t,x)$ is a strict maximum, we deduce that 
$(\tilde{t},\tilde{x})=(t,x)$.  
Therefore, it follows that $(s_m,y_m)\to (t,x)$ and $c_m\to 0$. 
By (\ref{eq:3.2}), for any $y$ near x, 
\begin{equation}
\label{eq:3.2.3}
 \varphi(s_m+h_m,y)+c_m+h_m^{3/2}\ge v^{h_m}(s_m+h_m,y). 
\end{equation}

Now, by Lemma \ref{lem:3.9}, we have  
\begin{equation*}
 \sum_{i=0}^2|D^iv^{h_m}(s_m+h_m,y_m)-D^iv^{h_m}(s_m,y_m)|\to 0, 
\end{equation*}
as $m\to\infty$. 
In particular, 
\begin{equation*}
 \lim_{m\to\infty}v^{h_m}(s_m+h_m,y_m)=\lim_{m\to\infty}v^{h_m}(s_m,y_m)
 =\varphi(t,x). 
\end{equation*}
Also, by Assumption \ref{assum:3.1}, 
\begin{equation}
\label{eq:3.2.5}
\begin{split}
&|F(s_m+h_m,y_m;v^{h_m}(s_m+h_m,\cdot)) - F(s_m,y_m;v^{h_m}(s_m,\cdot))| \\
&\le |F_0(s_m+h_m,y_m,v^{h_m}(s_m+h_m,y_m)) - F_0(t,x,\varphi(t,x))| \\
& \quad + |F_0(t,x,\varphi(t,x)) - F_0(s_m,y_m,v^{h_m}(s_m,y_m))| \\
&\quad + K_0|Dv^{h_m}(s_m+h_m,y_m)-Dv^{h_m}(s_m,y_m)| 
  +|D^2v^{h_m}(s_m+h_m,y_m)-D^2v^{h_m}(s_m,y_m)|.
\end{split}  
\end{equation}
As in the proof of Lemma \ref{lem:3.12}, rewrite $v^h(s_m,y_m)$ as 
\begin{equation}
\label{eq:3.3}
\begin{split}
 v^{h_m}(s_m,y_m) 
  &= v^{h_m}(s_m+h_m,y_m)
  -h_mF(s_m+h_m,y_m;v^{h_m}(s_m+h_m,\cdot)) \\
   &\quad +h_mR^{m}_1 + h_m R_2^m + h_mR^m_3, 
\end{split}
\end{equation}
where  
\begin{gather*}
 R^{m}_1=(1-\theta)\left(F(s_m,y_m;v^{h_m}(s_m,\cdot)) - 
   I_{F(s_m,\cdot;v^{h_m}(s_m,\cdot)),X}(y_m)\right), \\ 
  R_2^m = \theta\left(F(s_m+h_m,y_m;v^{h_m}(s_m+h_m,\cdot)) - 
    I_{F(s_m+h_m,\cdot;v^{h_m}(s_m+h_m,\cdot)),X}(y_m)\right), \\
  R_3^m = (1-\theta)\left(F(s_m+h_m,y_m;v^{h_m}(s_m+h_m,\cdot)) 
    - F(s_m,y_m;v^{h_m}(s_m,\cdot))\right). 
\end{gather*}
Assumption \ref{assum:3.6}, Lemma \ref{lem:3.9} 
and (\ref{eq:3.2.5}) guarantee 
$R^m_1, R_2^m, R_3^m\to 0$ as $m\to\infty$. 
With the representation (\ref{eq:3.3}), we apply Lemma \ref{lem:3.10} for the family 
$\{v^{h}(s_m+h,\cdot), \varphi(s_m+h,\cdot)\}_{h\in (0,1], m\ge 1}$   
and use the inequality (\ref{eq:3.2.3}) to get, for any sufficiently large $m$,  
\begin{align*}
 &v^{h_m}(s_m,y_m) \\
 &\le 
  \sup_{(p,\Gamma)\in\mathcal{D}_{h_m,\delta}}
  \inf_{w\in\mathcal{X}_{h_m,\kappa}}\Big[
 v^{h_m}(s_m+h_m,y_m+\sqrt{h_m}w)-\sqrt{h_m}p^{\mathsf{T}}w 
 -\frac{h_m}{2}w^{\mathsf{T}}\Gamma w \\
 &\qquad -h_mF(s_m+h_m,y_m,v^{h_m}(s_m+h_m,y_m),p,\Gamma)\Big] 
  +h_mR^{m}_1 + h_mR_2^m + h_mR_3^m 
   + Ch_m^{1+\beta} \\
 &\le \sup_{p,\Gamma}\inf_w\Big[
  \varphi(s_m+h_m,y_m+\sqrt{h_m}w)-\sqrt{h_m}p^{\mathsf{T}}w 
 -\frac{h_m}{2}w^{\mathsf{T}}\Gamma w \\
 &\qquad -h_mF(s_m+h_m,y_m,v^{h_m}(s_m+h_m,y_m),p,\Gamma)\Big] 
  +c_m+h_m^{3/2}+h_mR^{m}_1 
   + h_m R_2^m+h_m R_3^m + Ch_m^{1+\beta} \\
 &\le \varphi(s_m+h_m,y_m) 
         -h_mF(s_m+h_m,y_m,v^{h_m}(s_m+h_m,y_m),
     D\varphi(s_m+h_m,y_m),D^2\varphi(s_m+h_m,y_m)) \\
 &\qquad +c_m+h_m^{3/2}+h_m R^{m}_1 
    +h_mR_2^m+h_mR_3^m + Ch_m^{1+\beta}. 
\end{align*}
This and $v^{h_m}(s_m,y_m)=c_m+\varphi(s_m,y_m)$ imply 
\begin{equation}
\label{eq:3.4}
\begin{split}
 &-\frac{1}{h_m}\left(\varphi(s_m+h_m,y_m)-\varphi(s_m,y_m)\right) \\
 &\quad +F(s_m+h_m,y_m,v^{h_m}(s_m+h_m,y_m),D\varphi(s_m+h_m,y_m),
    D^2\varphi(s_m+h_m,y_m))
 \le o(1)
\end{split} 
\end{equation}
for any sufficiently large $m$.  
Letting $m\to\infty$, we arive at 
\begin{equation}
\label{eq:3.5}
 -\partial_t\varphi(t,x)+F(t,x,\overline{v}(t,x),D\varphi(t,x),D^2\varphi(t,x)) 
 \le 0. 
\end{equation}
Thus the subsolution property at $(t,x)$ follows. 

In the case $(t,x)\in\{T\}\times\mathbb{R}^d$, from 
the definition of $v^{h}$ and Assumption \ref{assum:3.6} we have 
$\overline{v}(t,x)=f(x)$. 
Thus the subsolution property immediately follows. 

Next consider the case $(t,x)\in [0,T)\times\partial\Omega$. 
As in the first part of the proof, we can take the sequence 
$(h_m,s_m,y_m)$, $m\ge 1$, satisfying (\ref{eq:3.2.3}) and 
$(s_m,y_m)\to (t,x)$. Moreover, 
$$
 v^{h_m}(s_m,y_m)=c_m+\varphi(s_m,y_m)\to \varphi(t,x)
 =\overline{v}(t,x). 
$$
Then, if there exists $m_0\ge 1$ such that 
$y_m\in\mathbb{R}^d\setminus\Omega$ for all $m\ge m_0$, we see 
$$
 v^{h_m}(s_m,y_m)=v(s_m,y_m)\to v(t,x), 
 \quad m\to\infty. 
$$
Thus the subsolution property follows. 
Otherwise, there exists a subsequence $\{y_{m_j}\}$ such that 
$y_{m_j}\in\Omega$ and $y_{m_j}\to x$, $j\to\infty$. 
With this sequence we obtain the inequality (\ref{eq:3.4}) 
with $(h_m,s_m,y_m)$ replaced by 
$(h_{m_j},s_{m_j},y_{m_j})$. 
Then letting $j\to\infty$, we obtain  
(\ref{eq:3.5}) at $(t,x)\in [0,T)\times\partial\Omega$.

By similar arguments, we can show that 
$$
\underline{v}(t,x)=\liminf_{{s\to t,\; y\to x}
 \atop{h\searrow 0}}v^{h}(s,y), 
\quad (t,x)\in [0,T]\times\mathbb{R}^d
$$
is a viscosity supersolution of (\ref{eq:1.1}). 
The comparison principle now implies that $\overline{v}\le \underline{v}$. 
However, by definition, $\overline{v}\ge \underline{v}$. 
Hence we obtain $\overline{v}=\underline{v}$, as asserted. 
\end{proof}

\section{A numerical example}\label{sec:4}

Here we consider the following two-dimensional 
deterministic KPZ equation
\begin{equation*}
\left\{
\begin{split}
 &\partial_t v + \frac{1}{2}\mathrm{tr}(D^2v) 
  +\frac{1}{2}\left|Dv\right|^2=0, \\
 & v(1,x)=f(x). 
\end{split}
\right. 
\end{equation*}
By Cole-Hopf transformation (see, e.g., Evans \cite{eva:1998}), the unique solution is 
represented as  
\begin{equation*}
 v(t,x)=\log\mathbb{E}\left[\exp\left( 
  f(x+W_{1-t})\right)\right], \quad (t,x)\in
 [0,1]\times\mathbb{R}^2, 
\end{equation*}
where $\{W_t\}_{0\le t\le  1}$ is a $2$-dimensional standard Brownian motion 
and $\mathbb{E}$ is the expectation operator on a probability space. 

We consider the case of the terminal data given by
\begin{equation*}
 f(x_1,x_2)=\cos(x_1)\cos(x_2)
\end{equation*}
and compute the solution in $\{0\}\times [-\pi/4,\pi/4]^2$ 
by our collocation method with $\theta=1$ and Gaussian RBF. 
We examine both the uniformly spaced grids and the Halton sequence 
on $[-\pi/2,\pi/2]^2$ consisting of $N$ points for 
the set $X$ of the data sites. 
Notice that we take the larger region $[-\pi/2,\pi/2]^2$ 
to expect a better performance near the boundary of $[-\pi/4,\pi/4]^2$. 
The adjustable parameter $\alpha$ for the kernel is set as 
$\alpha=1/\varepsilon^2$ where $\varepsilon$ is 
the Euclidean norm between the $N$ points in $[-\pi/2,\pi/2]^2$. 
As the benchmark, the exact solution $v(0,x)$ is estimated  
by the Monte-Carlo method with $10^6$ samples. 
Table \ref{table:4.1} shows the resulting root mean square errors and 
the maximum errors, defined by
\begin{equation*}
 \text{RMS error}=\sqrt{\frac{1}{625}\sum_{x\in X_0} 
 \left|v^{h}(0,x)-v(0,x)\right|^2}, \quad 
 \text{Max error} = \max_{x\in X_0}\left|v^{h}(0,x)-v(0,x)\right|,  
\end{equation*} 
respectively, where $X_0$ is the set of evaluation points consisting of 
$25^2$ uniformly spaced points in $[-\pi/4,\pi/4]^2$. 
\begin{table}[htb]
\begin{center}
\begin{tabular}[t]{cccccc} 
\toprule 
\addlinespace[3pt]
 & & \multicolumn{2}{c}{uniform} & \multicolumn{2}{c}{Halton}\\ \cmidrule(r){3-6}
 $N$  &  $h$ & Max error & RMS error & Max error & RMS error \\[1pt] \midrule 
 9  & 0.04 & 6.3177e-002 & 5.3287e-002 & 8.6257e-002 & 4.8638e-002  \\
     &  0.02 & 5.4872e-002 & 4.7082e-002 & 8.9769e-002 & 5.2891e-002 \\
      & 0.01  & 5.0207e-002 & 4.3563e-002 & 9.0458e-002 & 5.5374e-002 \\ \midrule
 16 & 0.04  & 3.4885e-003 & 1.2442e-003 & 5.2929e-002 & 2.0522e-002 \\ 
      & 0.02  & 9.2939e-003 & 7.3882e-003 & 5.5556e-002 & 2.4998e-002 \\  
      & 0.01  & 1.3321e-002 & 1.0278e-002 & 5.6317e-002 & 2.7392e-002 \\ \midrule 
 25 &  0.04 & 1.3885e-002 & 9.1823e-003 & 1.1283e-002 & 6.2947e-003 \\
      & 0.02 & 5.8901e-003 & 3.2270e-003 & 1.4613e-002 & 6.5674e-003 \\ 
      & 0.01  & 3.8536e-003 & 1.6292e-003 & 1.6812e-002 & 8.5034e-003 \\ \bottomrule  
\end{tabular}
\caption{RMS and Max errors in the cases of the uniformly spaced and Halton points 
for various choices of $N$ and $h$. }
\label{table:4.1}
\end{center}
\end{table}

\begin{figure}[htbp]
\centering
\subfigure{\includegraphics[width=0.45\columnwidth, bb = 0 0 600 400]{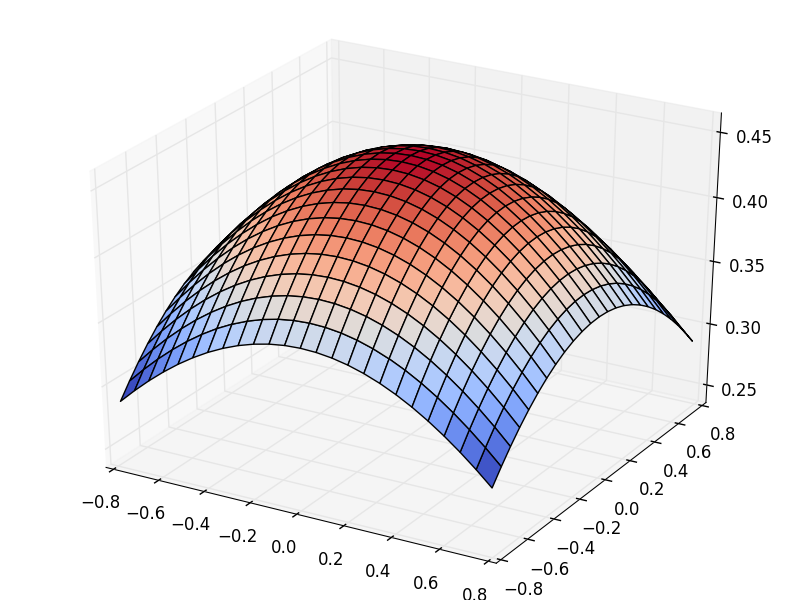}}
\subfigure{\includegraphics[width=0.45\columnwidth, bb = 0 0 600 400]{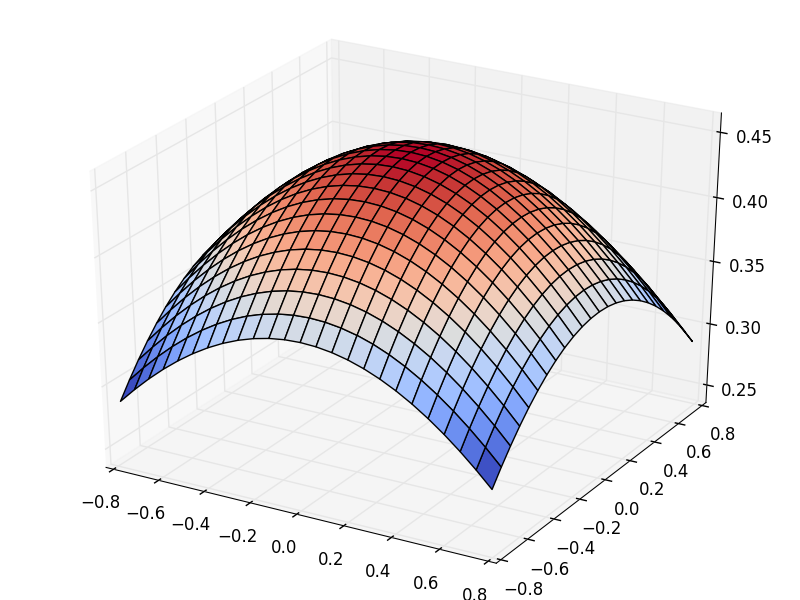}}
\caption{The analytical solution (left) and the numerical solution (right) with 
$h=10^{-2}$ and $N=25$ uniformly spaced points.}
\label{fig:4.1}
\end{figure}

\subsection*{Acknowledgements}

The author is thankful to the anonymous referees for their useful comments for  
previous versions of this paper. 
This study is partially supported by JSPS KAKENHI Grant Number 26800079.

\bibliographystyle{plain}
\bibliography{../mybib}

\end{document}